\documentclass[a4paper,11pt,reqno]{amsart}
\usepackage[bookmarks=false,colorlinks=true]{hyperref}
\usepackage{cite}
\usepackage{amsmath,amssymb}
\usepackage{tikz,pgfplots}
\pgfplotsset{compat=1.14} 
\usetikzlibrary{shapes,arrows,arrows.meta,cd}
\usepackage{empheq}

\def\Rnn{\R_{\geq 0}}

\def\K{\mathcal K}
\def\L{\mathcal L}
\def\KL{\K\L}
\def\Kinf{\K_\infty}

\def\Z{\mathbb{Z}}
\def\R{\mathbb{R}}

\def\X{\mathbb{X}}

\newcommand{\Rn}[1][n]{\R^{#1}}

\newcommand{\halmos}{\hfill \blacksquare}
\newcommand{\halmoswhite}{\hfill \square}
\newcommand{\nn}{{\nonumber}}

\newenvironment{DIFnomarkup}{}{}

\usepackage[normalem]{ulem}
\usepackage{color}
\usepackage{gensymb}

\newtheorem{theorem}{Theorem}
\newtheorem{corollary}[theorem]{Corollary}

\newtheorem{definition}[theorem]{Definition}
\newtheorem{claim}{Claim}

\newtheorem{remark}[theorem]{Remark}
\newtheorem{example}{Example}

\begin{document}

\title{Convergence Properties for Discrete-time Nonlinear Systems}

\author[D.~N.~Tran]{Duc N.\ Tran}
\address{Duc N.\ Tran and Christopher M.\ Kellett are with the School of Electrical Engineering and Computer Science, The University of Newcastle,\linebreak[4] Callaghan, NSW 2308, Australia.}
\email{DucNgocAnh.Tran@uon.edu.au}
\email{Chris.Kellett@newcastle.edu.au}
\author[B.~S.~R\"uffer]{Bj\"orn S.\ R\"uffer}
\address{Bj\"orn S.\ R\"uffer is with the School of Mathematical and Physical Sciences, The University of Newcastle, Callaghan, NSW 2308, Australia.}
\email{Bjorn.Ruffer@newcastle.edu.au}
\author[C.~M.~Kellett]{Christopher M.\ Kellett}

\begin{abstract}
  Three similar convergence notions are considered. Two of them are the long established notions of convergent dynamics and incremental stability. The other is the more recent notion of contraction analysis. All three convergence notions require that all solutions of a system  converge to each other. In this paper, we investigate the differences between these convergence properties for discrete-time, time-varying nonlinear systems by comparing the properties in pairs and using examples. We also demonstrate a time-varying smooth Lyapunov function characterization for each of these convergence notions, and, with appropriate assumptions, we provide several sufficient conditions to establish relationships between these properties in terms of Lyapunov functions.
\end{abstract}

\maketitle

\section{Introduction}

Convergence properties, namely \emph{incremental stability}~\cite{Yosh66,Ange02-TAC,ZaTa11-TAC}, \emph{convergent  dynamics}~\cite{Demidovich67,Pavlov02,PPWH04_SCL}, and \emph{contraction analysis}~\cite{LoSl98_AUT}, are established methods to characterize the asymptotic behavior of one solution with respect to any other solution of a (nonlinear) dynamical system. In particular, all convergence properties impose conditions that all solutions ``forget" their initial conditions and converge to each other. This is a highly desirable property in solving problems in nonlinear output stabilization and regulation \cite{Pavlov02,ZaTa11-TAC}, synchronization \cite{PhSl07_NN,CaCh06_PSMA,StSe07_TAC,Ange02-TAC}, observer design \cite{Ange02-TAC,LoSl98_AUT}, steady-state and frequency response analysis of nonlinear systems \cite{PaWo12_TAC,PWN07_TAC,PaWo08_ACC}, and others.
These three notions of convergence are similar and largely related to each other. However, the three methods were all derived independently, were motivated distinctly, and employ different tool sets. As a consequence, their mutual relationships are, in general, not fully understood even for continuous-time nonlinear systems in the literature. Thus, the purpose of this study is to investigate explicit differences between these properties.

Incremental stability describes the asymptotic property of differences between any two solutions.  Specifically, an augmented system is formed by two ``copies" of the original system. Then, global asymptotic stability of a special closed set (called the diagonal) with respect to the augmented system can be demonstrated to be equivalent to incremental stability of the original system. Thus, a Lyapunov characterization for incremental stability of the original system can be derived from a classical Lyapunov characterization of global asymptotic stability.  Historically, the origin of incremental stability goes back to works on forced oscillators by German mathematicians Trefftz \cite{trefftz1926-zu-den-grundlagen-der-schwingungstheorie} and Reissig \cite{reissig1955-zur-theorie-der-erzwungenen-schwingungen,reissig1955-uber-eine-nichtlineare-differentialgleichung-2.-ordnung,reissig1954-erzwungene-schwingungen-mit-zaher-und-trockener-reibung-erganzung}, without the property being named. It was LaSalle who in \cite{lasalle1957-a-study-of-synchronous-asymptotic-stability} attributes the term ``extreme stability'' to Reissig's work \cite{reissig1955-zur-theorie-der-erzwungenen-schwingungen}. Subsequently, the work of Yoshizawa \cite[Chapter IV]{Yosh66} made significant contributions to the study of extreme stability. In particular, for continuous-time, time-varying systems, a continuous Lyapunov function characterization of extreme stability was demonstrated\cite[Theorem~21.1]{Yosh66}. Similar results are found in \cite[Theorem~1]{Ange02-TAC} and \cite[Theorem~5]{RWM13-SCL} for time-invariant and time-varying systems, respectively, where in both works the property is termed incremental stability. Incremental stability for input-output operators is presented in \cite{FMN96-TAC}, and in \cite{fromionscorletti2002-the-behavior-of-incrementally-stable-discrete-time-systems} conclusions about (asymptotic) stability of equilibria and periodic solutions are drawn from linear incremental gain properties.  Incremental Input-to-State Stability \cite{Ange02-TAC} is an extension of incremental stability to systems with input using the Input-to-State Stability (ISS) framework~\cite{Sont89-TAC}.

Convergent dynamics, on the other hand, requires the existence of a unique and asymptotically stable reference solution that is bounded for all (backward and forward) time. However, a~priori knowledge of this reference solution is not necessarily required.  Every other solution, then, converges to this reference solution asymptotically. A convenient, sometimes easy to check, sufficient condition of convergent dynamics is the Demidovi{\v{c}} condition. Essentially, using a quadratic Lyapunov function, exponential convergence of all solutions to a reference solution can be demonstrated (see, for instance, \cite{PPWH04_SCL}). The existence and boundedness of the reference solution usually relies on Yakubovich's Lemma; i.e., that a compact and positively invariant set contains at least one bounded solution defined for all times. This result was demonstrated in \cite[Lemma 2]{Yaku64_ARC} based on ideas of Demidovi{\v{c}} in \cite{demidovic1961-on-the-dissipativity-of-a-certain-non-linear-system-of-differential-equations.-i, demidovic1962-on-the-dissipative-character-of-a-certain-non-linear-system-of-differential-equations.-ii}. Convergent dynamics (or convergent systems) was pioneered by Demidovi{\v{c}} \cite{Demidovich67} (only available in Russian), see also a historical perspective of developments for convergent dynamics in \cite{PPWH04_SCL}. Results on the convergent dynamics property for periodic systems are discussed in \cite{Pliss66}. A converse Lyapunov theorem for globally convergent systems was introduced in \cite[Theorem~7]{RWM13-SCL}. More details and further extensions of convergent dynamics to systems with input (input-to-state convergence) and applications to output regulation problems can be found in \cite{Pavlov02}.

Lastly, contraction analysis \cite{LoSl98_AUT} was inspired by fluid mechanics and differential geometry. Contraction analysis utilizes local analysis of the linearization along every trajectory to establish globally convergent behavior. This method, in essence, generalizes linear eigenvalue analysis to nonlinear systems. As a result, contraction analysis provides a characterization for exponential convergence of one solution with respect to any other solution. The contraction analysis approach, thus, independently extends the Demidovi{\v{c}} sufficient condition to a necessary and sufficient condition.  (Yet another approach, dating back at least to the 1950s, that allows linearizing nonlinear finite-dimensional systems into infinite-dimensional linear systems is via the Koopman operator~\cite{LaMe13-PDNP}---the adjoint of the transfer or Frobenius-Perron operator.)  Contraction analysis (or contraction metrics) was first introduced in \cite{LoSl98_AUT}. Since then, developments based on differential geometry have been made in \cite{FoSe12_TAC}, \cite{PoBu14_SCL} including converse theorems on contraction metrics for an equilibrium in \cite{Gies15-JMAA}. A historical perspective of contraction analysis, including earlier closely related concepts, is presented in \cite{Jouf05_CDC}.  Various applications of contraction analysis can be found in \cite{APS08_Aut,PhSl07_NN}.

In the continuous-time setting, several relationships have been established between these properties. Available comparisons include convergent dynamics versus incremental stability \cite{RWM13-SCL} and extending contraction analysis and incremental stability to the same differential geometric framework~\cite{FoSe12_TAC}. In \cite{fromionscorletti2005-connecting-nonlinear-incremental-lyapunov-stability-with-the-linearizations-lyapunov-stability} it is shown that exponential incremental stability is equivalent to exponential (asymptotic) stability of one and of all trajectories of a system, and how this relates to the linearization of the system, assuming the state space is closed and convex.

In this paper, we study and compare the three previously discussed convergence properties for time-varying nonlinear systems, albeit, in contrast to most of the existing literature, in discrete-time. In particular, we show that convergent dynamics and incremental stability are two distinct properties. Similarly, convergent dynamics and contraction analysis are two distinct properties. However, contraction analysis is a special case of incremental stability, namely, exponential incremental stability. Furthermore, we establish various sufficient conditions to demonstrate relationships between the three convergence properties. This is done by either tightening the regularity requirements on the dynamics or by adding an assumption on the state space. We also provide time-dependent smooth Lyapunov function characterizations for each of the properties. Furthermore, we present discrete-time analogues of the Demidovi{\v{c}} condition for incremental stability (similar to the continuous-time result \cite[Equation 15]{RWM13-SCL}), convergent dynamics (similar to the continuous-time result \cite[Theorem~1]{PPWH04_SCL}), and contraction analysis.

To derive the three Lyapunov function characterizations, we rely heavily on a converse Lyapunov result for time-varying systems with an asymptotically stable set  \cite[Theorem~1]{JiWa02-SCL} and a stronger version of this theorem for time-varying systems with an exponentially stable set \cite[Theorem~2]{JiWa02-SCL}. 
Specifically,  we use a result \cite[Lemma 2.3]{Ange02-TAC} to convert global asymptotic (or exponential) incremental stability to global asymptotic (or exponential) stability of a closed set with respect to an augmented system. For convergent dynamics, we first apply a nonlinear change of coordinates with respect to the reference solution, then we use a standard
Lyapunov result \cite[Theorem~1]{JiWa02-SCL} (similar to \cite[Theorem~23]{Mass56-AM} for the continuous-time case) to obtain the desired characterization. Lastly, we demonstrate that contraction analysis, in fact, is equivalent to exponential incremental stability, and hence, is equivalent to the existence of an exponential incremental Lyapunov function.

To demonstrate the differences between the three considered convergence notions, we construct examples of systems that individually satisfy one of the three properties but not the others. Specifically, we exploit the fact that global
incremental stability (any two trajectories tend to each other asymptotically)
is a stronger property than global asymptotic convergence to a single trajectory. From there, we are able to construct a system (Example~\ref{example:1}) showing that convergent dynamics does not imply incremental stability. On the other hand, we exploit another fact of convergent systems: that there must exist at least one bounded solution. This is not necessarily true for asymptotically incrementally stable systems. Thus, we show that incremental stability does not imply convergent dynamics (Example~\ref{example:2}).

In contrast to the other two properties, contraction analysis explicitly requires continuously differentiable dynamics. 
Given a system with differentiable right hand side, we are able to show that contraction analysis is equivalent to exponential incremental stability (Theorem~\ref{equivalent contraction vs incremental}). As a consequence, contraction analysis does not imply convergent dynamics. Finally, we employ the fact that the convergence rate in contraction analysis is always exponential. This is not necessarily true for convergent systems (Example~\ref{example:3}). Therefore, contraction analysis and convergent dynamics are distinct convergence notions.

The paper is organized as follows: the necessary technical
assumptions, notational conventions, and definitions of convergence
properties are provided in Section~\ref{sec:pre}. Then, Lyapunov
function characterizations for these notions are presented in
Section~\ref{incre compare sec: Lyapunov}.  In Section~\ref{incre
  compare sec:comparison}, comparisons in pairs for the three
properties are presented together with various examples and sufficient
conditions.  Conclusions and future research indications are provided
in Section~\ref{incre compare sec:Conc}. 
Several proofs are collected in the appendix.

\section{preliminaries}\label{sec:pre}
Let $\Z$ denote the set of integers and $\Z_{\geq k_{0}}$ the set of
integers greater or equal to $k_{0}$.  We consider discrete-time
nonlinear time-varying systems described by the difference equation
\begin{equation}
  x(k+1)=f(k,x(k)),\quad  x(k)\in\R^n,k\in\mathbb{Z},  \label{system no input}
\end{equation}
where $f\colon \Z\times\R^n\rightarrow\R^n$ is continuous and $x(k_0)=\xi\in\R^n$ for $k_0\in\Z$.  For any $k_{0}\in\Z$ and $\xi\in\R^n$, the solution of system~\eqref{system no input} is a function $\phi\colon  \Z_{\geq k_{0}}\rightarrow\R^n$, parameterized by initial state and time, and satisfying $\phi(k_0;k_0,\xi)=\xi$ and equation~\eqref{system no input}; i.e., $\phi(k+1;k_0,\xi)=f(k,\phi(k;k_0,\xi))$ for all $k,k_0\in\Z$ such that $k\geq k_0$.  We use standard comparison function classes\footnote[1]{Recall that $\alpha\colon  \mathbb{R}_{\geq 0}\rightarrow\mathbb{R}_{\geq 0}$ is of class-$\mathcal{K} $ if it is continuous, zero at zero, and strictly increasing. If $\alpha\in \mathcal{K} $ is unbounded, it is of class-$\mathcal{K_\infty}$. A function $\sigma\colon  \mathbb{R}_{\geq 0}\rightarrow\mathbb{R}_{\geq 0}$ is of class-$\mathcal{L}$ if it is continuous, strictly decreasing, and $\lim_{t\rightarrow\infty} \sigma(t) =0$. A function $\beta\colon  \mathbb{R}_{\geq 0}\times\mathbb{R}_{\geq 0}\rightarrow\mathbb{R}_{\geq 0}$ is of class-$\mathcal{K}\mathcal{L}$ if it is class-$\mathcal{K}$ in its first argument and class-$\mathcal{L}$ in its second argument. By convention, $\beta \in\mathcal{K}\mathcal{L}$ satisfies $\beta(0,t)=0$ for all $t\in\R_{\geq 0}$.}
$\mathcal{K}$, $\mathcal{L}$, $\mathcal{K_\infty}$, and $\mathcal{K}\mathcal{L}$ (see \cite{Kell14-MCSS}).  For a vector $y \in \mathbb{R}^n$, a matrix $P\in\R^{n\times n}$, and a positive definite matrix $Q$ we denote the Euclidean norm by $|y|$, the induced (matrix) norm by $\displaystyle\|P\|=\max_{x\in\R^n,~x\neq 0}\frac{|Px|}{|x|}$, and the induced norm of $y$ with respect to matrix $Q$ by $\|y\|_Q=\sqrt{y^TQy}$.
Given two symmetric matrices $A$ and $B$ in $\R^{n\times n}$ we write $A\preceq B$ if for all $x\in\R^{n}$, $x^{T}Ax\leq x^{T}Bx$. 
We say that a set $\X\subseteq\R^n$ is positively invariant under~\eqref{system no input} if for all $\xi\in\X$ and $k,k_0\in \Z$ such that $k\geq k_0$, we have $\phi(k;k_0,\xi)\in\X$.

In the following subsections, we recall the definitions of incremental stability, convergent dynamics, and contraction analysis.
\subsection{Incremental Stability}

The following definition of asymptotic incremental stability is a discrete-time analogue to that in \cite[Definition 2.1]{Ange02-TAC}.

\begin{definition}\label{incremental stable def} System~\eqref{system no input} is \emph{uniformly asymptotically incrementally stable} in a positively invariant set $\X \subseteq \R^n$ if there exists $\beta \in \mathcal{K}\mathcal{L}$ such that
  \begin{equation}
    |\phi(k;k_0,\xi_1)-\phi(k;k_0,\xi_2)|\leq \beta(|\xi_1-\xi_2|,k-k_0),\label{incremental stable}
  \end{equation}
  holds for all $\xi_1,\xi_2\in\X$ and $k,k_0\in \Z$ such that $k\geq k_0$. In case $\X=\R^n$, we say that system~\eqref{system no input} is \emph{uniformly globally asymptotically incrementally stable}.
\end{definition}

As is standard, \emph{uniform} here refers to the fact that the bound in \eqref{incremental stable} depends only on the elapsed time $k-k_0$. A strictly stronger property requires an exponential rate of convergence.

\begin{definition}\label{exponential incremental stable def} System~\eqref{system no input} is  \emph{uniformly exponentially incrementally  stable} in a positively invariant set $\X \subseteq \R^n$ if there exist $\kappa \geq 1$ and $\lambda>1$ such that
  \begin{align}
    |\phi(k;k_0,\xi_1)-\phi(k;k_0,\xi_2)| \leq \kappa|\xi_1 -\xi_2|\lambda^{- (k-k_0)}\label{exponential incremental stability}
  \end{align}	
  holds for all $\xi_1,\xi_2\in\X$ and $k,k_0\in \Z$ such that $k\geq k_0$. In case $\X=\R^n$, we say that system~\eqref{system no input} is \emph{uniformly globally exponentially incrementally stable}.
\end{definition}	

\subsection{Convergent Dynamics}
The following definition of convergent dynamics for discrete-time systems is recalled from \cite[Definition 1]{PaWo12_TAC}. This definition is a discrete-time analogue to the continuous-time definition in \cite[Definition 1]{PPWH04_SCL}. 

\begin{definition}\label{convergent dynamics def}
  System~\eqref{system no input} is \emph{uniformly convergent} in a positively invariant set $\X \subseteq \R^n$ if 

  \begin{enumerate}
  \item there exists a unique bounded solution $\bar{x}(k)$ of system~\eqref{system no input} defined in $\X$ for all $k\in \Z$;  
  \item there exists a function $\beta \in \mathcal{K}\mathcal{L}$ such that, 
    for all $\xi \in  \X$, $k,k_0\in \Z$ with $k\geq k_0$, we have
    \begin{equation}
      |\phi(k;k_0,\xi)-\bar{x}(k)|\leq \beta(|\xi-\bar{x}(k_0)|,k-k_0).\label{convergent system}
    \end{equation}
  \end{enumerate}
  In case $\X=\R^n$, we say that system~\eqref{system no input} is \emph{uniformly globally convergent}. As in the previous definition, we say
  system~\eqref{system no input} is uniformly (globally) \emph{exponentially} convergent if the $\KL$-function $\beta$ can be chosen to
  be $\beta(s,k-k_{0})=\kappa s\lambda^{-(k-k_{0})}$ for some $\kappa\geq 1$ and $\lambda>1$.
\end{definition}

\begin{remark}
  In continuous-time, the definition of convergent dynamics \cite[Definition 1]{PPWH04_SCL} requires an additional condition that system~\eqref{system no input} is forward complete. In discrete-time, such a condition is not required because, by definition, the range of $f$ is $\R^{n}$.
\end{remark}

\begin{remark}\label{unique of bar x}
  In fact, condition $2)$ and the boundedness of $\bar x(k)$ in Definition~\ref{convergent dynamics def} also imply that $\bar x(k)$ is unique. Indeed, let $\hat x(k)$ be another solution defined and bounded for all $k\in\Z$, then $|\hat x(k_0)-\bar x(k_0)|$ is bounded for all $k_0\in\Z$. Since~\eqref{convergent system} holds for all solutions, taking the limit for $k_0 \rightarrow -\infty$ in~\eqref{convergent system}, it follows that $|\hat x(k)-\bar x(k)|\leq 0$ for all $k\in\Z$. Hence, $\hat x\equiv \bar x$.
\end{remark}

\subsection{Contraction Analysis}
Unlike the previous two properties, contraction analysis explicitly requires the dynamics of the considered system to be continuously differentiable. As a consequence, in this subsection, we make a standing assumption that the mapping $f$ (or the right hand side) of system~\eqref{system no input} is continuously differentiable in $x$ on $\R^n$.

We are now ready to state the formal definition (following \cite{LoSl98_AUT}) of contraction analysis.
\begin{definition}
  \label{def:contraction}
  Suppose the right hand side of system~\eqref{system no input} is continuously differentiable in $x$ on a positively invariant set $\X \subseteq \R^n$. Then system~\eqref{system no input} is \emph{uniformly contracting} in $\X$
  if there exist a nonsingular-matrix-valued function $\Theta\colon  \Z\times\X \rightarrow \R^{n\times n}$ and constants $\mu, \eta,\rho >0$ such that, for all $x \in  \X, k\in \Z$, we have
  \begin{align}
    &\eta I \preceq \Theta(k,x)^T\Theta(k,x)\preceq \rho I, \label{bounds on Theta}\\
    &F(k,x)^TF(k,x)-I\preceq-\mu I, \label{contraction condition}
  \end{align}
  where the matrix $F(k,x)$ is given by
  \begin{align}
    F(k,x)=\Theta(k+1,x)\frac{\partial f}{\partial x}(k,x) \Theta (k,x)^{-1}.\label{define Jacobian matrix}
  \end{align}
  In case $\X=\R^n$, we say that system~\eqref{system no input} is \emph{uniformly globally contracting}.
\end{definition}
\begin{remark}
  While the previous definitions for convergence properties are all
  trajectory-based, Definition~\ref{def:contraction} is presented in a
  manner consistent with the literature in terms of the existence of
  the matrix-valued function $\Theta$.  As we will demonstrate in
  Theorem~\ref{equivalent contraction vs incremental} in the sequel,
  contraction analysis is equivalent to uniform exponential
  incremental stability.  Hence, an alternate definition of
  contraction analysis involves a simple nomenclature change in
  Definition~\ref{exponential incremental stable def}.
\end{remark}
\begin{remark}
  Note that~\cite[Definition 3]{LoSl98_AUT}
  includes~\eqref{contraction condition} but not explicitly both
  bounds in~\eqref{bounds on Theta}.  The lower bound in \eqref{bounds
    on Theta} is required to guarantee contraction as the following
  example illustrates. The upper bound guarantees uniformity of the
  contraction property with respect to initial time, which is
  generally desirable.  Consider the system
  \begin{align}
    x(k+1) = (k^2+1)x(k), \quad x(k)\in\R^n, k\in\Z. \label{counterexample contraction}
  \end{align}
  Choose $\Theta(k,x)\coloneqq \frac{1}{k^2+1}$, then
  $F(k,x)=\frac{1}{k^2+2k+2}$ and, thus,\linebreak[4]
  $F(k,x)^TF(k,x) =\frac{1}{(k^2+2k+2)^2} <\frac{1}{4}$; i.e.,
  condition~\eqref{contraction condition} is satisfied.
  However, the origin is an equilibrium and is unstable by inspection.
  Thus, this system is not uniformly globally
  contracting. Note that for this particular choice of $\Theta(k,x)$,
  the lower bound of condition~\eqref{bounds on Theta} is not
  satisfied.
\end{remark}

Given~\eqref{system no input}, similar to \cite{FoSe12_TAC}, we define the variational system 
\begin{subequations}
  \label{eq:tot}
  \begin{empheq}[left={}\empheqlbrace]{align}
    x(k+1)&=f(k,x(k)),\label{original dynamics}\\
    x_\delta(k+1)&= \frac{\partial f}{\partial x} (k,x(k)) x_\delta(k).\label{displacement dynamics}
  \end{empheq}
\end{subequations}
As in \cite{LoSl98_AUT}, we refer to~\eqref{displacement dynamics} as the displacement dynamics where $x_\delta$ is called a displacement of $x$ and is sometimes denoted by $\delta x$ in the literature.

\section{Lyapunov function characterizations}\label{incre compare sec: Lyapunov}
In this section, we present time-varying Lyapunov function characterizations for the three considered convergence properties.

\subsection{Incremental Stability}
A Lyapunov function characterization of incremental stability for continuous-time systems was first presented in \cite[Theorem~21.1 \& Theorem 21.2]{Yosh66}. Subsequently, similar results for time-invariant and time-varying systems were given in \cite[Theorem~1]{Ange02-TAC} and \cite[Theorem~5]{RWM13-SCL}, respectively. The following theorem is a discrete-time analogue to \cite[Theorem~5]{RWM13-SCL}. By a smooth function we mean one that is infinitely often differentiable.

\begin{theorem}\label{incre Lyapunov characterization}
  System~\eqref{system no input} is uniformly globally asymptotically incrementally stable if and only if there exist a smooth function $V\colon  \Z\times\R^n\times\R^n\rightarrow\R_{\geq 0}$,  functions $\alpha_1,\alpha_2 \in \Kinf$, and $\alpha_{3}\colon\Rnn\to\Rnn$ positive definite, such that
  \begin{align}
    \alpha_1(|x_1-x_2|) \leq V(k,x_1,x_2)&\leq \alpha_2(|x_1-x_2|),\label{inceremental Lyapunov function 1}\\
    V\Big(k+1,f(k,x_1),f(k,x_2)\Big)-V(k,x_1,x_2)
    &\leq -\alpha_3(|x_1-x_2|)
          \label{inceremental Lyapunov function 2}
  \end{align}
  hold for all $x_1,x_2 \in \mathbb{R}^n$ and $k\in \Z$. 
\end{theorem}

We remark that, in accordance with~\cite[Lemma~2.8]{JiWa02-SCL}, the function $\alpha_{3}$ provided by the above result may be assumed to be of class~$\Kinf$ without loss of generality.

The proof of Theorem~\ref{incre Lyapunov characterization} is contained in the appendix. 
A function $V$ satisfying~\eqref{inceremental Lyapunov function 1}--\eqref{inceremental Lyapunov function 2} is called an incremental stability Lyapunov function.

\subsection{Convergent Dynamics}
The following theorem is a time-varying Lyapunov function characterization of discrete-time globally convergent systems. This result is analogous to a continuous-time result presented in \cite[Theorem~7]{RWM13-SCL}.
\begin{theorem}\label{conv Lyapunov characterization}
  Assume that system~\eqref{system no input} is uniformly globally convergent. Then, there exist a smooth function $V\colon  \Z\times\mathbb{R}^n\rightarrow\mathbb{R}_{\geq 0}$, a constant $c\geq 0$, and functions $\alpha_1,\alpha_2, \alpha_3 \in \mathcal{K_\infty}$ such that
  \begin{align}
    \alpha_1(|x-\bar x(k)|) \leq V(k,x)&\leq \alpha_2(|x-\bar x(k)|),\label{convergent Lyapunov function 1}\\
    V\Big(k+1,f(k,x)\Big)-V(k,x) &\leq -\alpha_3(|x-\bar x(k)|),\label{convergent Lyapunov function 2}\\
    V(k,0)\leq c &<+\infty\label{convergent Lyapunov function 3}
  \end{align}
  hold for all $x \in \mathbb{R}^n$ and $k\in \Z$. 
  Conversely, if a smooth function $V\colon  \Z\times\mathbb{R}^n\rightarrow\mathbb{R}_{\geq 0}$, a constant $c\geq 0$,  functions $\alpha_1,\alpha_2, \in \Kinf$, and $\alpha_{3}\colon\Rnn\to\Rnn$ positive definite  are given such that for some trajectory $\bar x\colon  \Z \rightarrow \R^n$ estimates~\eqref{convergent Lyapunov function 1}--\eqref{convergent Lyapunov function 3} hold, then system~\eqref{system no input} is uniformly globally convergent and the solution $\bar x(k)$ is the unique bounded solution as in Definition~\ref{convergent dynamics def}. The same statement holds for uniform global exponential convergence if the $\mathcal{K}$ functions $\alpha_{1}$, $\alpha_{2}$, and $\alpha_{3}$ are replaced by quadratic functions.
\end{theorem}
The proof of Theorem~\ref{conv Lyapunov characterization} is contained in the appendix. 
A function $V$ satisfying~\eqref{convergent Lyapunov function 1}--\eqref{convergent Lyapunov function 3} is called a convergent dynamics Lyapunov function.

\subsection{Contraction Analysis}

The following theorem is a Lyapunov function characterization for globally contracting systems. Note that this result is a special case of the Lyapunov function characterization for uniform global asymptotic incremental stability.

\begin{theorem}\label{cont Lyapunov characterization}
  Assume the right hand side $f$ of~\eqref{system no input} is continuously differentiable in $x$.
  System~\eqref{system no input} is globally contracting if and only if there exist a smooth function $V\colon  \Z\times\mathbb{R}^n\times\mathbb{R}^n\rightarrow\mathbb{R}_{\geq 0}$, and constants $c_1,c_2, c_3 \in \R_{>0}$ such that
  \begin{align}
    c_1|x_1-x_2|^2 \leq V(k,x_1,x_2)&\leq c_2|x_1-x_2|^2,\label{contraction Lyapunov function 1}\\
    V\Big(k+1,f(k,x_1),f(k,x_2)\Big)-V(k,x_1,x_2)
    &\leq -c_3|x_1-x_2|^2
          \label{contraction Lyapunov function 2}
  \end{align}
  hold for all $x_1,x_2 \in \mathbb{R}^n$ and $k\in \Z$. 
\end{theorem}

The proof of Theorem~\ref{cont Lyapunov characterization} is contained
in the appendix.
A function $V$ satisfying~\eqref{contraction Lyapunov function
  1}--\eqref{contraction Lyapunov function 2} is called a contraction
analysis Lyapunov function.

\section{Comparisons}\label{incre compare sec:comparison}

In this section, we compare the three previously discussed properties in pairs. While any asymptotically stable linear system satisfies all three properties, there are distinct differences when nonlinear systems are considered. We provide several examples of systems that highlight the essential differences among the three convergence notions. Several sufficient conditions are also proposed to establish mutual relationships. Finally, we provide a discrete-time Demidovi{\v{c}} condition that is sufficient for all three properties.

\subsection{Incremental Stability vs. Convergent Dynamics}
The following example, which is adapted from \cite[Example~3]{RWM13-SCL}, provides a system that is uniformly globally convergent but not uniformly globally asymptotically incrementally stable.

\begin{example}
  \label{example:1}
  Consider the system
  \begin{equation}
    \label{IncSt vs ConDy ex 1}
    x(k+1)=f(k,x(k)),\qquad x(k) \in\R^{2},
  \end{equation}
  with
  \begin{align*}
    f(k,x)\coloneqq
    \frac{1}{2} 
    \begin{bmatrix}
      \cos(|z(k)|) &-\sin(|z(k)|) \\
      \sin(|z(k)|) & \cos(|z(k)|)
    \end{bmatrix}
                     z(k)
                     +
                     \begin{bmatrix}
                       \cos(k+1) \\
                       \sin(k+1)
                     \end{bmatrix},
  \end{align*}
  where $\bar{x}(k) \coloneqq \big(\cos(k), \sin(k)\big)^{T}$ and $z(k)\coloneqq x(k)-\bar{x}(k)$. 
  Clearly, $\bar{x}(k)$ is a bounded solution of \eqref{IncSt vs ConDy
    ex 1}. Now, consider the time-varying, quadratic Lyapunov function
  $V(k,x)=\big|z(k)\big|^2$. We compute
  \begin{multline*}
    V\big(k+1,f(x)\big)-V(k,x)
    =\big|f(k,x)-\bar x(k+1)\big|^2-\big|z(k)\big|^2\\
    =\frac{1}{4}\big|z(k)\big|^2-\big|z(k)\big|^2
    =-\frac{3}{4}\big|z(k)\big|^2\leq 0.
  \end{multline*}
  Hence, appealing to Theorem~\ref{conv Lyapunov characterization}, system~\eqref{IncSt vs ConDy ex 1} is uniformly globally convergent.
  Rewriting $z(k)$ in polar coordinates $\big(r(k),\theta(k)\big)$ yields 
  \begin{align*}
    &r(k+1)=\frac{r(k)}{2}, &
    &\theta(k+1)=\theta(k)+r(k),\\
    \intertext{whose explicit solution for an initial condition $(r_0,\theta_0)\in\R\times\R$ is}
    &r(k)=\frac{r_0}{2^{k-k_{0}}}, & 
    &\theta(k)=\theta_0+ r_0\sum_{\kappa=0}^{k-k_{0}-1}\frac{1}{2^{\kappa}}.
  \end{align*}

  \begin{center}
    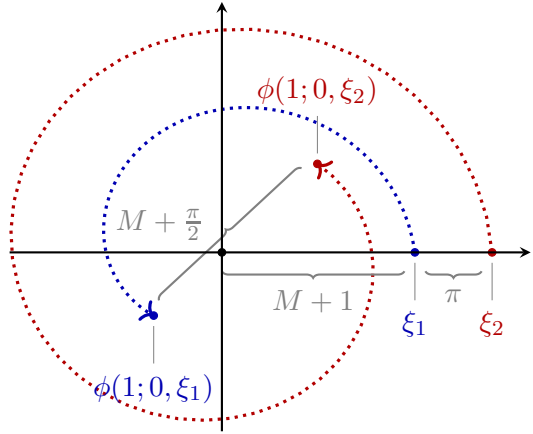
\begin{figure}
      \centering
      \begin{DIFnomarkup}
        \begin{tikzpicture}[baseline=5.5cm]
          \colorlet{primary}{blue!70!black}
          \colorlet{secondary}{red!70!black}
          \begin{axis}[
            domain=-7:7,
            y domain=-7:7,
            view={0}{90},
            xmin=-5.5,xmax=8,ymin=-5,ymax=7,
            axis lines=middle,
            axis line style={->,thick,-stealth,black},
            xtick=\empty,ytick=\empty,
            ]
            \addplot3+[secondary,domain=0:1,samples=50,samples y=0,mark=none,dotted,very thick,->]
            ( {(7-3.5*x)*cos(x*405)}, {(7-3.5*x)*sin(x*405)}, {0} ) ;
            \addplot3+[primary,domain=0:1,samples=50,samples y=0,mark=none,dotted,very thick,->]
            ( {(5-2.5*x)*cos(x*225)}, {(5-2.5*x)*sin(x*225)}, {0} ) ;
            \draw[fill] (axis cs: 0,0) circle[radius=.5mm];
            \draw[primary,fill] (axis cs: 5,0) circle[radius=.5mm] node[pin=-90:{$\xi_{1}$}] (xi1) {};
            \draw[secondary,fill] (axis cs: 7,0) circle[radius=.5mm] node[pin=-90:{$\xi_{2}$}] (xi2) {};
            \draw[primary,fill] (axis cs: -1.767,-1.767) circle[radius=.5mm] node[pin=-90:{$\phi(1;0,\xi_{1})$}] (x1) {};
            \draw[secondary,fill] (axis cs: 2.474,2.474) circle[radius=.5mm] node[pin=90:{$\phi(1;0,\xi_{2})$}] (x2) {};

            \draw [
            gray,
            thick,
            decoration={
              brace,
              raise=0.1cm
            },
            decorate
            ] (x1) -- (x2) 
            node [pos=0.5,anchor=north,yshift=0.55cm,xshift=-1cm] {$M+\frac{\pi}{2}$};

            \draw [
            gray,
            thick,
            decoration={
              brace,
              mirror,
              raise=0.25cm
            },
            decorate
            ] (axis cs:0,0) -- (xi1) 
            node [pos=0.5,anchor=north,yshift=-0.35cm] {$M+1$};

            \draw [
            gray,
            thick,
            decoration={
              brace,
              mirror,
              raise=0.25cm
            },
            decorate
            ] (xi1) -- (xi2) 
            node [pos=0.5,anchor=north,yshift=-0.35cm] {$\pi$};
            
          \end{axis}
        \end{tikzpicture}
      \end{DIFnomarkup}
      
      \caption{The two trajectories of the system defined in Example~\ref{example:1} start on the positive real half line with an
        initial separation of $\pi$
        at time $k = 0$ and the initial distances to the
        origin are $M+1$ and $M+1+\pi$. At time $k = 1$,
        the polar arguments of the trajectories are shifted by $180\degree$ so that the separation distance
        is $M+\frac{\pi}{2}$.
        \label{incre counterexample 1}
      }
    \end{figure}
  \end{center}

  At time $k_{0}=0$ consider two initial conditions in Cartesian
  coordinates $\xi_{1}=(M+1,0)$ and $\xi_{2}=(M+\pi+1,0)$ for some
  $M \in \R_{>0}$ as shown in Figure~\ref{incre counterexample 1}. The
  polar coordinates of $z(0)$ corresponding to $\xi_1$ and $\xi_2$ are
  $(r_1^{0},\theta_1^{0})=(M,0)$ and
  $(r_2^{0},\theta_2^{0})=(M+\pi,0)$. The initial separation is
  $|\xi_{1}-\xi_{2}|=\pi$. At the time instance $k=1$, the polar
  coordinates of $z(1)$ are $\big(r_1(1),\theta_1(1)\big)=\big(\frac{M}{2},M\big)$ and
  $(r_2(1),\theta_2(1))=\big(\frac{M+\pi}{2},M+\pi\big)$, hence, the
  angle difference is $\pi$. Consequently, the distance between two
  states is
  $|\phi(1;0,\xi_{1})-\phi(1;0,\xi_{2})|=r_1(1)+r_2(1)=\frac{M}{2}+\frac{M+\pi}{2}=M+\frac{\pi}{2}$.
  Since $M\in \R_{>0}$ is arbitrary, for any function $\beta\in \KL$,
  we can always choose a sufficiently large $M$ such
  that~\eqref{incremental stable} is violated. Thus,
  system~\eqref{IncSt vs ConDy ex 1} is not incrementally
  stable.$\halmoswhite$
\end{example}

The above example exploits the fact that global asymptotic convergence of any two trajectories to each other (incremental stability) is a stronger property than global asymptotic convergence of all trajectories to a single trajectory (convergent dynamics). This is a direct consequence of the triangle inequality.

The following theorem provides a connection from convergent dynamics to incremental stability by restricting the state space of system~\eqref{system no input} to be a compact and positively invariant set.

\begin{theorem}\label{convergent to incremental}
  Suppose system~\eqref{system no input} is uniformly convergent on a compact and positively invariant set $\X \subset\R^n$.  Further assume that the right hand side $f$ of system~\eqref{system no input} is locally Lipschitz continuous in $x$ on $\X$. Then, system~\eqref{system no input} is uniformly asymptotically incrementally stable on $\X$. Moreover, if the attraction rate in the uniform convergence property is exponential, then the system is uniformly exponentially incrementally stable.
\end{theorem}

The proof of Theorem~\ref{convergent to incremental} is contained in
the appendix.  Conversely, incremental stability does not imply
convergent dynamics as shown by the following example adapted from
\cite[Example~4]{RWM13-SCL}.
\begin{example}
  \label{example:2}
  The system
  \begin{align}
    x(k+1)=-\frac{k}{2}-1+\frac{x(k)}{2}, ~~\text{with }  x(k_0)=\xi\in \mathbb{R} \label{ConDy Incre example 2}
  \end{align}
  has the explicit solution
  \begin{align}
    \phi(k;k_0,\xi)= \frac{k_0}{2^{k-k_0}}-k+\frac{\xi}{2^{k-k_0}},\quad\forall k\geq k_0.\label{ConDy Incre example 2 solution}
  \end{align}
  For any two initial conditions $\xi_1,\xi_2\in\R$,
  \begin{align*}
    &|\phi(k;k_0,\xi_1)-\phi(k;k_0,\xi_2)|
      = \frac{|\xi_1-\xi_2|}{2^{k-k_0}}=\beta(|\xi_1-\xi_2|,k-k_0),
  \end{align*}
  where $\beta \in \mathcal{K}\mathcal{L}$ is defined by $\beta(s,r)=\frac{s}{2^r}$ for all $s,r\in\mathbb{R}_{\geq 0}$. Hence, system~\eqref{ConDy Incre example 2} is uniformly globally asymptotically incrementally stable.
  However, for the specific initial condition $(\xi,k_0)=(0,0)$, it is straightforward that the solution~\eqref{ConDy Incre example 2 solution} from this initial condition is unbounded. Therefore, system~\eqref{ConDy Incre example 2} is not uniformly globally convergent.$\halmoswhite$
\end{example}

The above example exploits the fact that convergent dynamics requires the existence of at least one bounded solution. Whereas solutions of an incrementally stable system can be unbounded.

The following theorem provides a connection from incremental stability to convergent dynamics with an assumption that the state space of system~\eqref{system no input} is compact and positively invariant.

\begin{theorem}\label{Theorem Yakubovich lemma application}
  Suppose system~\eqref{system no input} is uniformly globally asymptotically incrementally stable and there exists a compact set $\X\subset \R^n$ that is positively invariant under~\eqref{system no input}. Then system~\eqref{system no input} is uniformly globally convergent. Moreover, if the incremental attraction rate is exponential, then it is also exponential for the convergence property.
\end{theorem}

The proof of Theorem~\ref{Theorem Yakubovich lemma application} is contained in the appendix.

The following theorem provides a sufficient condition for system~\eqref{system no input} to satisfy all three convergence properties.  It is the discrete-time analogue of the Demidovi{\v{c}} result \cite[Theorem~1]{PPWH04_SCL}.

\begin{theorem}[Discrete-time Demidovi{\v{c}} condition]
  \label{Demidovich condition theorem}
  Assume the right hand side $f$ of system~\eqref{system no input} is continuously differentiable in $x$. Suppose there exists a positive definite matrix $P$ such that the matrix
  \begin{align}
    J(k,x)\coloneqq \frac{\partial f}{\partial x}(k,x)^T P \frac{\partial f}{\partial x}(k,x)-\rho P \label{Demidovich condition}
  \end{align}
  is negative semidefinite uniformly in $(k,x)\in\Z\times\R^n$ for some $\rho \in (0,1)$. Then system~\eqref{system no input} is uniformly globally exponentially incrementally stable and globally contracting. Furthermore, if there exists $c\geq0$ such that
  \begin{align}
    \sup_{k\in\Z}|f(k,0)|=c<\infty, \label{Demidovich bounded condition}
  \end{align}
  then system~\eqref{system no input} is uniformly globally exponentially convergent.
\end{theorem}

The proof of Theorem~\ref{Demidovich condition theorem} is contained in the appendix.

\subsection{Contraction Analysis vs. Incremental Stability}

Assuming the right hand side of system~\eqref{system no input} is continuously differentiable, we demonstrate that contraction analysis is a strictly stronger property than asymptotic incremental stability. In fact, contraction analysis is equivalent to exponential incremental stability.
\begin{theorem}\label{equivalent contraction vs incremental}
  Suppose the right hand side of system~\eqref{system no input} is continuously differentiable in $x$ on $\R^n$. System~\eqref{system no input} is uniformly globally exponentially incrementally stable if and only if it is uniformly globally contracting.
\end{theorem}

The proof of Theorem~\ref{equivalent contraction vs incremental} is contained in the appendix.

\subsection{Convergent Dynamics vs. Contraction Analysis}

System~\eqref{IncSt vs ConDy ex 1} in Example~\ref{example:1} is globally convergent, however, it is not globally contracting because it is not differentiable at $x=0$. Hence, a uniformly globally convergent system is not necessarily globally contracting.

However, even in the case where the system dynamics are continuously differentiable, global convergent dynamics is still different from contraction analysis. This is due to the fact that contraction analysis requires an exponential convergence rate whereas convergent dynamics only requires asymptotic convergence. The following example, adapted from \cite[Example~1]{TKD15_Aut}, provides a (time-invariant) system that is asymptotically convergent but not exponentially convergent.

\begin{example}
  \label{example:3}
  The system
  \begin{align}
    \begin{aligned}
      x(k+1)=f(x(k)) \coloneqq \frac{x(k)}{\sqrt{x(k)^2+1}},
    \end{aligned}
    \label{conv eq example 3}
  \end{align}
  with $x(k_0)=\xi\in\mathbb{R}$
  has the solution
  \begin{align}
    \phi(k;k_{0},\xi) = \frac{\xi}{\sqrt{(k-k_0)\xi^2+1}},\quad \forall k \in\Z, k\geq k_0 . \label{conv eq example 3 solution}
  \end{align}
  The zero solution $\bar x(k)=0$ is (uniformly) globally asymptotically stable. Indeed, for the Lyapunov function $ V(x)=x^2$, we have, for all $x \neq 0$,
  \begin{align}
    V(f(x)) -V(x) =-\frac{x^4}{x^2+1}< 0. \label{conv example 1:1}
  \end{align}
  Hence, system~\eqref{conv eq example 3} is (uniformly) globally convergent. By inspection of~\eqref{conv eq example 3 solution}, we see that the convergence to the zero solution is not exponential, hence, by Theorem~\ref{equivalent contraction vs incremental}, system~\eqref{conv eq example 3} is not globally contracting.$\halmoswhite$
\end{example}

\begin{example} 
  Returning to Example~\ref{example:2}, system~\eqref{ConDy Incre example 2} is not convergent since the solution passing through $(k_0,\xi)=(0,0)$ is unbounded. However, system~\eqref{ConDy Incre example 2} is globally contracting. Indeed, take $\Theta (k,x)=1$ for all $k\in\Z$ and $x\in\R$, $\mu=\frac{3}{4}$, and note that $\frac{\partial f}{\partial x}(k,x)=\frac{1}{2}$ for all $k\in \Z$ and $x\in\mathbb{R}^n$. With $F(k,x)$ given by~\eqref{define Jacobian matrix} we see that 
  \begin{align*}
    &F(k,x)^TF(k,x)-1 = -\mu.
  \end{align*}
  Hence, system~\eqref{ConDy Incre example 2} is globally contracting.$\halmoswhite$
\end{example}

Using the results of Theorem~\ref{Theorem Yakubovich lemma application} and Theorem~\ref{equivalent contraction vs incremental}, we provide sufficient conditions under which global contraction implies convergent dynamics.

\begin{theorem}\label{Convergence vs contraction theorem}
  Suppose the mapping $f$ of system~\eqref{system no input} is continuously differentiable in $x$ on a compact and positively invariant set $\X\subset\R^n$. If system~\eqref{system no input} is uniformly contracting in $\X$, then system~\eqref{system no input} is uniformly exponentially convergent in $\X$.  
\end{theorem}

The proof of Theorem~\ref{Convergence vs contraction theorem} is
contained in the appendix.

\section{Conclusions}\label{incre compare sec:Conc}

This paper contributes discrete-time, time-varying, smooth Lyapunov function characterizations for incremental stability, convergent dynamics, and contraction analysis. The paper also contributes examples of systems that highlight the essential differences as well as similarities among the three considered notions of stability. Moreover, with appropriate assumptions, we present several conditions that provide connections between each of the three considered convergence properties. Overall, assuming the right hand side of the considered system is continuously differentiable, the relationships between the three convergence properties can be summarized as in Figure~\ref{Convergence commutative diagram}.

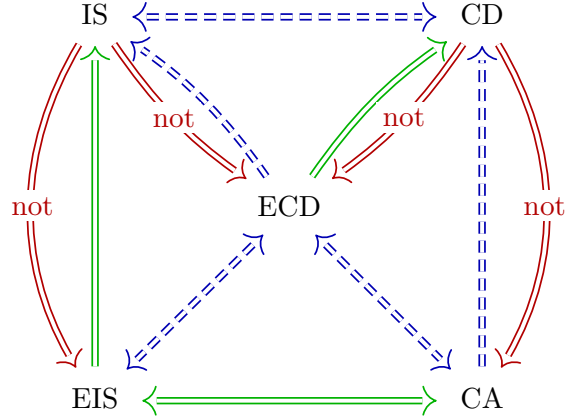
\begin{figure}
  \centering
  \newcommand{\myfontsize}[1]{\mbox{#1}}
  \colorlet{falsecolor}{red!70!black}
  \colorlet{truecolor}{green!70!black}
  \colorlet{compactcolor}{blue!70!black}
  \begin{DIFnomarkup}
    \begin{tikzcd}[ row sep=2cm, column sep=1.5cm, 
      arrow style=tikz,
      arrows={
        /tikz/line width=.65pt,
        Rightarrow,
        double,
        >={.To[scale=0.7]},
        shorten <=4pt,
        shorten >=4pt,
      }
      ]
      \text{IS}
      \arrow[rr, compactcolor, dashed, <->]
      \arrow[dd,"{\myfontsize{not}}" description, ->, bend right,falsecolor]
      && \text{CD}
      \arrow[dd,"{\myfontsize{not}}" description, ->, bend left,
      falsecolor]
      \\ &
      \text{ECD}
      \arrow[ul, ->, bend right=10,compactcolor, dashed]
      \arrow[ul,"{\myfontsize{not}}" description, <-, bend left=10,falsecolor]
      \arrow[ur, ->, bend left=10,truecolor]
      \arrow[ur,"{\myfontsize{not}}" description, <-, bend right=10,falsecolor]
      \arrow[dl, <->,compactcolor, dashed]
      \arrow[dr, <->,compactcolor, dashed]
      & \\%
      \text{EIS}
      \arrow[rr, <->, truecolor]
      \arrow[uu, -> ,truecolor]
      && \text{CA}
      \arrow[uu, ->, compactcolor, dashed]
    \end{tikzcd}
  \end{DIFnomarkup}
  \caption{Relationships between different convergence properties
    assuming continuously differential dynamics $f$. Note the
    abbreviations: IS for (global asymptotic) \emph{Incremental
      Stability}, EIS for (global) \emph{Exponential Incremental
      Stability}, CD for (uniform global) \emph{Convergent
      Dynamics}, ECD for (uniform global) \emph{Exponential
      Convergent Dynamics}, and CA for (global) \emph{Contraction
      Analysis}.  The blue, dashed implications do not hold in
    general, but for the case of a compact (and invariant) state space
    $\X$.}\label{Convergence commutative diagram}
\end{figure}

\section*{Appendix}

As several proofs in this appendix rely on the same modification of a converse
Lyapunov result originally due to Jiang and Wang\cite{JiWa02-SCL},
we state the modification here explicitly for the convenience of the
reader.

\begin{definition}[cf.~\cite{JiWa02-SCL}]
  Let $\mathcal{A}$ be a closed, not necessarily compact, positively
  invariant set of~\eqref{system no input}. System~\eqref{system no
    input} is \emph{uniformly global asymptotically stable} (UGAS)
  with respect to $\mathcal{A}$ if there exists a class-$\KL$ function $\beta$ such that
  \begin{equation}
    \label{eq:8}
    |\phi(k;k_{0},\xi_{0})|_{\mathcal{A}} \leq \beta(|\xi_{0}|_{\mathcal{A}}, k-k_{0})
  \end{equation}
  for all $k\geq k_{0}$ and all $\xi_{0}\in\Rn$, where 
  $$
  |\xi|_{\mathcal{A}} \coloneqq \inf_{\eta\in\Rn}|\eta-\xi|
  $$
  denotes the distance of $\xi$ to $\mathcal{A}$.    
\end{definition}

Similarly, the notion of uniform global \emph{exponential} stability
(UGES) is defined by replacing estimate~\eqref{eq:8} with the
exponential bound
\begin{equation}
  \label{eq:9}
  |\phi(k;k_{0},\xi_{0})|_{\mathcal{A}} \leq  \kappa|\xi_{0}|_{\mathcal{A}}\lambda^{- (k-k_0)},
\end{equation}
for some constants $\kappa\geq1$ and $\lambda>1$.

\begin{theorem}[Jiang \& Wang 2002]
  \label{thm:converse-Lyapunov-UGAS}
  Assume system~\eqref{system no input} is UGAS with respect to the
  closed set $\mathcal{A}$. Then there exist a smooth function
  $V\colon\Z\times\Rn\to\R_{\geq0}$ and functions
  $\alpha_{1},\alpha_{2}, \alpha_{3} \in\Kinf$, such that for all
  $k\in\Z$ and all $x\in\Rn$,
  \begin{equation}
    \label{eq:4}
    \alpha_{1}(|x|_{\mathcal{A}})\leq V(k,x) \leq \alpha_{2}(|x|_{\mathcal{A}})
  \end{equation}
  and
  \begin{equation}
    \label{eq:5}
    V\big(k+1,f(k,x)\big) - V(k,x) \leq -\alpha_{3}(|x|_{\mathcal{A}}).
  \end{equation}
  Conversely, if there exist a continuous function
  $V\colon\Z\times\Rn\to\R_{\geq0}$, functions
  $\alpha_{1},\alpha_{2} \in\Kinf$, and a positive definite function
  $\alpha_{3}\colon\R_{\geq0}\to\R_{\geq0}$, such that \eqref{eq:4}--\eqref{eq:5}
  hold, then system~\eqref{system no input} is UGAS with respect to
  the closed set $\mathcal{A}$.
\end{theorem}

The original result in~\cite[Theorem 1]{JiWa02-SCL} asserts the
existence of a function $V$ defined on $\Z_{\geq 0}$
satisfying~\eqref{eq:4} and, for a positive definite function
$\alpha_{3}$, \eqref{eq:5}. In contrast, the formulation above
yields a function $V$ defined on all of $\Z$. That the function
$\alpha_{3}$ can indeed be chosen to be of class-$\Kinf$ instead of
merely positive definite is already proven in~\cite[Lemma
2.8]{JiWa02-SCL}.

It is thus no surprise that the proof of the above result follows
the original very closely.

\begin{proof}
  Start by assuming that system~\eqref{system no input} is UGAS with
  respect to the closed set $\mathcal{A}$.  Applying Sontag's Lemma
  on $\KL$-estimates \cite[Proposition 7]{Sont98-SCL}, given
  $\beta \in \mathcal{K}\mathcal{L}$, there exist
  $\rho_1,\rho^{-1} \in \mathcal{K_\infty}$ such that
  \begin{equation}
    \beta(|\xi_{0}|_{\mathcal{A}},k-k_0)\leq \rho^{-1}(\rho_1(|\xi_{0}|_{\mathcal{A}})e^{-(k-k_0)})
  \end{equation}
  for all $ \xi_{0}\in\R^n$ and $k,k_0\in\Z$ such that $k\geq
  k_0$. Hence,
  \begin{equation}  
    \rho(|\phi(k;k_0,\xi_{0})|_{\mathcal{A}})\leq \rho_1(|\xi_{0}|_{\mathcal{A}})e^{-(k-k_0)}. \label{converse: Sontag Lemma conseq}
  \end{equation}
  Define $V\colon \Z\times\Rn\rightarrow\R_{\geq0}$ by
  \begin{align}
    V(k_0,\xi_{0}) \coloneqq \sum_{k = k_0}^{\infty}\rho(|\phi(k;k_0,\xi_{0})|_{\mathcal{A}}).\label{converse: construct V}
  \end{align}
  It is true that for all $k_0\in\Z$ and $\xi_{0}\in\R^n$
  \begin{align}
    \rho(|\xi_{0}|_{\mathcal{A}})& =\rho(|\phi(k_0;k_0,\xi_{0})|_{\mathcal{A}}) \leq V(k_0,\xi_{0})\nonumber\\ 
                                 &\leq \sum_{k = k_0}^{\infty}\rho_1(|\xi_{0}|_{\mathcal{A}})e^{-(k-k_0)}\leq \frac{e}{e-1} \rho_1(|\xi_{0}|)
  \end{align}
  where the second inequality follows from \eqref{converse: Sontag
    Lemma conseq} and \eqref{converse: construct V}. This means that
  the series defined in \eqref{converse: construct V} is
  convergent. Moreover, for each $k_0\in \Z$, $\rho$ is a continuous
  function in $\xi_{0}$, thus, $V$ is continuous in $\xi_{0}$. 

  Define $\alpha_1 \coloneqq \rho$ and
  $\alpha_2(s)\coloneqq \frac{e}{e-1} \rho_1(s)$ for all
  $s\in\R_{\geq 0}$. Observe that $\alpha_{1}$ and $\alpha_{2}$ are
  both of class-$\Kinf$.  We immediately see that $V$ satisfies
  condition~\eqref{eq:4}.

  It is only left to prove that $V$ satisfies
  condition~\eqref{eq:5}. Indeed, for arbitrary $k_{0}\in\Z$ and
  $\xi_{0}\in\Rn$ we have
  \begin{align*}
    V(k_0+1,f(k_0,\xi_{0})) &= \sum_{k = k_0+1}^{\infty}\rho(|\phi(k;k_0+1,f(k_0,\xi_{0}))|_{\mathcal{A}})\\
                            &= \sum_{k = k_0+1}^{\infty}\rho(|\phi(k;k_0+1,f(k_0,\xi_{0}))|_{\mathcal{A}}) \\
                            & \quad + \rho(|\phi(k_0;k_0,\xi_{0})|_{\mathcal{A}}) -\rho(|\phi(k_0;k_0,\xi_{0})|_{\mathcal{A}})\\
                            &= \sum_{k = k_0}^{\infty}\rho(|\phi(k;k_0,\xi_{0})|) -\rho(|\phi(k_0;k_0,\xi_{0})|_{\mathcal{A}})\\
                            &= V(k_0,\xi_{0}) - \rho(|\xi_{0}|_{\mathcal{A}}).
  \end{align*}
  We conclude that the continuous function $V$
  satisfies~\eqref{eq:5} with $\alpha_{3}\coloneqq\rho$.  Lastly, an
  application of \cite[Lemma~2.7]{JiWa02-SCL} allows us to replace
  the continuous function $V$ by a smooth one (which also modifies
  the $\alpha_{i}$, $i=1,2,3$).

  The converse direction is standard and the same as in
  \cite[Section~4.3.3]{JiWa02-SCL} and thus omitted.
\end{proof}

\begin{corollary}
  \label{corollary:converse-Lyapunov-UGES}
  If system~\eqref{system no input} is UGES with respect to the
  closed set $\mathcal{A}$, then there exist a smooth function
  $V\colon\Z\times\Rn\to\R_{\geq0}$ and constants
  $c_{1},c_{2}, c_{3}>0$, such that for all
  $k\in\Z$ and all $x\in\Rn$,
  \begin{equation}
    \label{eq:6}
    c_{1}|x|_{\mathcal{A}}^{2}\leq V(k,x) \leq c_{2}|x|_{\mathcal{A}}^{2}
  \end{equation}
  and
  \begin{equation}
    \label{eq:7}
    V\big(k+1,f(k,x)\big) - V(k,x) \leq -c_{3}|x|_{\mathcal{A}}^{2}.
  \end{equation}
  Conversely, the existence of a continuous function
  $V\colon\Z\times\Rn\to\R_{\geq0}$ and positive constants
  $c_{1},c_{2}, c_{3}$ satisfying~\eqref{eq:6}--\eqref{eq:7} is
  sufficient for system~\eqref{system no input} to be UGES with
  respect to the closed set $\mathcal{A}$.
\end{corollary}

\begin{proof}
  This result follows as in the previous proof by observing that in the
  additional step of using Sontag's Lemma on $\KL$-estimates is not
  necessary due to the exponential estimate~\eqref{eq:9}, which is equivalent to
  \begin{equation*}
    |\phi(k;k_{0},\xi_{0})|_{\mathcal{A}}^{2} \leq  \kappa^{2}|\xi_{0}|_{\mathcal{A}}^{2}\big(\lambda^{2}\big)^{- (k-k_0)}.
    \tag*{\qedhere}      
  \end{equation*}    
\end{proof}

With these two results at hand, we can now turn to proving the main
results in this paper.

\subsection*{Proof of Theorem~\ref{incre Lyapunov characterization}} \label{proof of incre Lyapunov charac}
Consider the augmented system
\begin{equation}
  \begin{aligned}
    \begin{cases} 
      x_1(k+1)=f(k,x_1(k)) \\ 
      x_2(k+1)=f(k,x_2(k)) 
    \end{cases}
  \end{aligned}\label{incre augmented system}
\end{equation}
as in \cite{Ange02-TAC}.  The \textit{diagonal} is the set  $\Delta\coloneqq \{[x^T,x^T]^T\in R^{2n}\colon x\in \R^n\}$. Let $z=\big[\begin{smallmatrix}x_1 \\ x_2\end{smallmatrix}\big]\in \R^{2n}$, $x_1,x_2\in\R^n$, then the distance from $z$ to the diagonal is given by
\begin{align}
  |z|_\Delta\coloneqq \inf_{w\in \Delta} |w-z|=\frac{1}{\sqrt{2}}|x_1-x_2|,\label{copy system relation}
\end{align}
where the equality is shown in \cite{Ange02-TAC}. Denote $F\Big(k,\big[\begin{smallmatrix}x_1 \\ x_2\end{smallmatrix}\big]\Big)\coloneqq  \Big[\begin{smallmatrix}f(k,x_1(k)) \\ f(k,x_2(k))\end{smallmatrix}\Big]$ and rewrite~\eqref{incre augmented system} as
\begin{align}
  z(k+1)=F(k,z(k)),\label{incre augmented system 1}
\end{align}
where $z=\big[\begin{smallmatrix}x_1 \\ x_2\end{smallmatrix}\big]\in \R^{2n}$ and $k\in\Z$.

Appealing to the relationship~\eqref{copy system relation}, we see
that system~\eqref{system no input} is uniformly globally
asymptotically incrementally stable if and only if system~\eqref{incre
  augmented system 1} is uniformly globally asymptotically stable
with respect to the diagonal set $\Delta$.

By application of Theorem~\ref{thm:converse-Lyapunov-UGAS},
system~\eqref{incre augmented system 1} admits a smooth, time-varying
Lyapunov function
$V\colon \Z\times\mathbb{R}^{2n}\rightarrow\mathbb{R}_{\geq 0}$ for
which there exist functions
$\alpha_1,\alpha_2, \alpha_3 \in \mathcal{K_\infty}$ such that
\begin{align}
  &~~\alpha_1(|z|_\Delta) \leq V(k,z)\leq \alpha_2(|z|_\Delta),\label{inceremental augmented Lyapunov function 1}\\
  &V\Big(k,F(k,z)\Big)-~V(k,z) \leq -\alpha_3(|z|_\Delta)\label{inceremental augmented Lyapunov function 2}
\end{align}
hold for all $z \in \mathbb{R}^{2n}$ and $k\in \Z$.  Using~\eqref{copy
  system relation}, it is clear that~\eqref{inceremental augmented
  Lyapunov function 1}--\eqref{inceremental augmented Lyapunov
  function 2} are correspondingly equivalent to~\eqref{inceremental
  Lyapunov function 1}--\eqref{inceremental Lyapunov function 2}.

Again using Theorem~\ref{thm:converse-Lyapunov-UGAS}, we conclude that
system~\eqref{system no input} is uniformly globally asymptotically
incrementally stable if and only if system~\eqref{system no input}
admits a Lyapunov function satisfying~\eqref{inceremental Lyapunov
  function 1}--\eqref{inceremental Lyapunov function 2}. $\halmos$

\subsection*{Proof of Theorem~\ref{conv Lyapunov characterization}} \label{proof of conv Lyapunov charac}
Assume that system~\eqref{system no input} is uniformly globally
convergent as in Definition~\ref{convergent dynamics def}, so that
there exists a solution $\bar x(k)$ bounded for all $k\in\Z$; i.e.,
$\sup_{k\in\Z}|\bar x(k)| <\infty$.

After a change of coordinates $z(k)=x(k)-\bar x(k)$ to the original system, we obtain a new system
\begin{equation*}
  z(k+1)=f\big(x(k)\big)-f\big(\bar x(k)\big)=f\big(z(k)+\bar x(k)\big)-f\big(\bar x(k)\big)
  \eqqcolon g\big(k,z(k)\big).
\end{equation*}
Applying Theorem~\ref{thm:converse-Lyapunov-UGAS} to this new system yields a smooth
function $W\colon \Z\times\mathbb{R}^n\rightarrow\mathbb{R}_{\geq 0}$
and functions $\alpha_1,\alpha_2, \alpha_3 \in \mathcal{K_\infty}$
such that
\begin{align}
  &\alpha_1(|z|) \leq W(k,z)\leq \alpha_2(|z|),\text{ and} \label{eq:2}\\
  &W\big(k+1,g(k,z)\big)-W(k,z) \leq -\alpha_3(|z|),\label{eq:3}
\end{align}
hold for all $z \in \Rn$ and $k\in \Z$. Reverting to the original
coordinates and defining $V(k,x)\coloneqq W\big(k,x-\bar x(k)\big)$, we
obtain~\eqref{convergent Lyapunov function 1}--\eqref{convergent
  Lyapunov function 2} from \eqref{eq:2}--\eqref{eq:3}.
Moreover,
\begin{align*}
  0\leq V(k,0)\leq \alpha_2(|\bar x(k)|) \leq \alpha_2\Big(\sup_{k\in\Z}|\bar x(k)|\Big)\eqqcolon c<\infty
\end{align*}
establishes~\eqref{convergent Lyapunov function 3}.  This completes
the proof of the first statement of the theorem.

Conversely, if~\eqref{convergent Lyapunov function
  1}--\eqref{convergent Lyapunov function 3} hold, then with the same
coordinate change and by applying the forward result in
Theorem~\ref{thm:converse-Lyapunov-UGAS}, we see that
system~\eqref{system no input} is uniformly globally asymptotically
stable with respect to the trajectory $\bar x(k)$. Thus, it is only
left to show that $\bar x(k)$ is bounded and unique. Indeed,
\begin{align*}
  |\bar x(k)|\leq \alpha_1^{-1}(V(k,0))\leq \alpha_1^{-1}(c)<\infty
\end{align*}
shows that $\bar x(k)$ is bounded. Lastly, uniqueness of $\bar x(k)$
follows from Remark~\ref{unique of bar x}.

Replacing every invocation of Theorem~\ref{thm:converse-Lyapunov-UGAS}
in this proof by an invocation of
Corollary~\ref{corollary:converse-Lyapunov-UGES} proves the statement
about uniform global exponential convergence.  $\halmos$

\subsection*{Proof of Theorem~\ref{cont Lyapunov characterization}} \label{proof of cont Lyapunov charac}	

The proof proceeds by demonstrating the equivalence of an exponential incremental Lyapunov function
and exponential incremental stability and then appealing to Theorem~\ref{equivalent contraction vs
  incremental}.

Similar to the proof of Theorem~\ref{incre Lyapunov characterization},
consider the augmented system~\eqref{incre augmented system} (also,
its shorthand notation~\eqref{incre augmented system 1}). 
Since~\eqref{copy system relation} holds, system~\eqref{system no
  input} is uniformly globally exponentially incrementally stable if
and only if the diagonal set $\Delta$ is uniformly globally
exponentially stable with respect to system~\eqref{incre augmented
  system 1}.

Applying Corollary~\ref{corollary:converse-Lyapunov-UGES}, the closed set
$\Delta$ is uniformly globally exponentially stable with respect to
system~\eqref{incre augmented system 1} if and only if
system~\eqref{incre augmented system 1} admits a smooth,
time-varying Lyapunov function
$V\colon  \Z\times\mathbb{R}^{2n}\rightarrow\mathbb{R}_{\geq 0}$ for which
there exist constants $c_1,c_2, c_3 \in \R_{>0}$ such that
\begin{align}
  &c_1|z|_\Delta^{2} \leq V(k,z)\leq c_2|z|_\Delta^{2},\label{inceremental expo augmented Lyapunov function 1}\\
  &V\big(k,F(k,z)\big)-~V(k,z) \leq -c_3|z|_\Delta^{2}\label{inceremental expo augmented Lyapunov function 2}
\end{align}
hold for all $z \in \mathbb{R}^{2n}$ and $k\in \Z$. Using~\eqref{copy
  system relation}, it is clear that~\eqref{inceremental expo
  augmented Lyapunov function 1}--\eqref{inceremental expo augmented
  Lyapunov function 2} are correspondingly equivalent
to~\eqref{contraction Lyapunov function 1}--\eqref{contraction
  Lyapunov function 2}.  Therefore, we conclude that
system~\eqref{system no input} admits an (exponential incremental)
Lyapunov function~\eqref{contraction Lyapunov function
  1}--\eqref{contraction Lyapunov function 2} if and only if
system~\eqref{system no input} is uniformly globally exponentially
incrementally stable.

Applying Theorem~\ref{equivalent contraction vs
  incremental} yields that system~\eqref{system no input} is uniformly
globally exponentially incrementally stable if and only if
system~\eqref{system no input} is globally contracting.  Note that the
proof of Theorem~\ref{equivalent contraction vs incremental} does not
employ any results prior to Theorem~\ref{cont Lyapunov
  characterization} or Theorem~\ref{cont Lyapunov characterization}
itself.
This completes the proof.$\halmos$

\subsection*{Proof of Theorem~\ref{convergent to incremental}}	\label{Proof of convergence implies incremental}
The proof follows that of \cite[Theorem~8]{RWM13-SCL}. Denote the diameter of $\X$ as $d_\X\coloneqq \max_{x,y\in\X}|x-y|$. 
Fix any $\xi_1,\xi_2\in\X$. Applying the triangle inequality, we see that
\begin{align}
  &|\phi(k;k_0,\xi_1)-\phi(k;k_0,\xi_2)|\nn\\
  &\leq |\phi(k;k_0,\xi_1)-\bar x(k)|+|\phi(k;k_0,\xi_1)-\bar x(k)|\nn\\
  &\leq  2\beta(d_\X,k-k_0)\label{Proof of convergent implies incremental 1}
\end{align}	
for all $x\in\X$ and $k,k_0\in \Z$ such that $k\geq k_0$. Let $L\in\R_{>0}$ be the maximum Lipschitz constant for $f$ on $\X$; i.e.,
\begin{align}
  |f(k_0,\xi_1)-f(k_0,\xi_2)|\leq L|\xi_1-\xi_2|\label{Proof of convergent implies incremental 2}
\end{align}
for all $\xi_1,\xi_2\in\X$. It immediately follows from~\eqref{Proof of convergent implies incremental 2} that
\begin{align}
  |\phi(k;k_0,\xi_1)-\phi(k;k_0,\xi_2)| \leq L^{k-k_0}|\xi_1-\xi_2|. \label{Proof of convergent implies incremental 3}
\end{align}
Combining~\eqref{Proof of convergent implies incremental 1} and~\eqref{Proof of convergent implies incremental 2}, we see that
\begin{equation*}
  |\phi(k;k_0,\xi_1)-\phi(k;k_0,\xi_2)|
  \leq \min\big\{L^{k-k_0}|\xi_1-\xi_2|,2\beta(d_\X,k-k_0)\big\}.
\end{equation*}
From there, we can obtain a function $\hat\beta\in \mathcal{K}\mathcal{L}$ such that
\begin{align*}
  |\phi(k;k_0,\xi_1)-\phi(k;k_0,\xi_2)| \leq\hat\beta (|\xi_1-\xi_2|,k-k_0)\nn
\end{align*}
for all $\xi_1,\xi_2\in\X$ and $k,k_0\in \Z$ such that $k\geq k_0$.
Therefore, system~\eqref{system no input} is uniformly asymptotically incrementally stable on $\X$.

With only minor rewording of the proof we establish the assertion about
exponential incremental stability.  $\halmos$

\subsection*{Proof of Theorem~\ref{Theorem Yakubovich lemma application}}
\label{proof of Yakubovich lemma application}	

Suppose system~\eqref{system no input} is uniformly globally
asymptotically incrementally stable, then by applying
Theorem~\ref{incre Lyapunov characterization}, there exists a function
$V$ satisfying~\eqref{inceremental Lyapunov function
  1}--\eqref{inceremental Lyapunov function 2}.

To derive the existence of a bounded solution $\bar x(k)$, we use a
discrete-time version from \cite[Lemma 2]{PaWo08_ACC} of
Yakubovich's~Lemma~\cite[Lemma~2]{Yaku64_ARC} which states if there
exists a compact and positively invariant set $\X$ of
system~\eqref{system no input}, then there exists a solution
$\bar x(k)$ of system~\eqref{system no input} defined on $\Z$ and
satisfying $\bar x(k)\in \X$ for all $k\in\Z$. Thus, we see that
\begin{align*}
  \alpha_1(|x-\bar x(k)|) &\leq \strut V(k,x,\bar x(k))\leq \alpha_2(|x-\bar x(k)|)\\
  \intertext{and}
  V\Big(k+1,f(k,x),f(k,\bar x(k))\Big)&-V(k,x,\bar x(k)) \leq -\alpha_3(|x-\bar x(k)|)\
\end{align*}
for all $x\in\R^n$ and $k\in\Z$.  Moreover, since $\bar x(k)$ is bounded, we have
\begin{align}
  V(k,0,\bar x(k))\leq  \alpha_2(|0-\bar x(k)|) < \infty
\end{align}
for all $k\in\Z$.

Hence, the function $W(k,x)\coloneqq V(k,x,\bar x(k))$ satisfies
conditions~\eqref{convergent Lyapunov function 1}--\eqref{convergent
  Lyapunov function 3}. Thus, by applying Theorem~\ref{conv Lyapunov
  characterization}, we conclude that system~\eqref{system no input}
is uniformly globally convergent.  $\halmos$

\subsection*{Proof of Theorem~\ref{Demidovich condition theorem}}

Consider any two initial conditions $\xi_1,\xi_2\in\R^n$. Define a function $\Phi\colon [0,1]\rightarrow\R $ by
\begin{align}
  \Phi (s)\coloneqq \big(f(k,\xi_1) - f(k,\xi_2)\big)^TP f\left(k,s\xi_1+(1-s)\xi_2\right).\label{convex combination 1}
\end{align}
Then, $\Phi$ satisfies
\begin{align}
&\Phi(1) - \Phi(0) =\left(f(k,\xi_1) - f(k,\xi_2)\right)^TP  \left(f(k,\xi_1) - f(k,\xi_2)\right).\label{Demidovich condition -1} 
\end{align} 
Applying the mean value theorem, there exists an $\bar s\in[0,1]$ such that
\begin{multline}
  \label{Demidovich condition 0}
  \Phi(1)-\Phi(0) = \frac{d}{ds}\Phi(\bar s)\\
  = \big(f(k,\xi_1) - f(k,\xi_2)\big)^TP\frac{\partial f}{\partial x}\left(k,\bar s\xi_1+(1-\bar s)\xi_2\right)(\xi_1-\xi_2).
\end{multline}
Denote $\check\xi\coloneqq \bar s \xi_1+(1-\bar s)\xi_2$.
From \eqref{Demidovich condition -1}, \eqref{Demidovich condition 0}, we obtain
\begin{multline}
  \label{Demidovich condition 1.5}
  \big(f(k,\xi_1) - f(k,\xi_2)\big)^TP  \big(f(k,\xi_1) -f(k,\xi_2)\big)  \\
  =\big(f(k,\xi_1) - f(k,\xi_2)\big)^TP\frac{\partial f}{\partial x}(k,\check \xi)(\xi_1-\xi_2)
\end{multline}
We make the following claim:
\begin{claim} For all $\xi_1,\xi_2\in\R^n$ and $k\in\Z$, the following holds
\begin{align}
  &\left(f(k,\xi_1) - f(k,\xi_2)\right)^TP  \left(f(k,\xi_1) -f(k,\xi_2)\right)\nn\\
  &\leq (\xi_1-\xi_2)^T \left(\frac{\partial f}{\partial x}(k,\check \xi)^T P \frac{\partial f}{\partial x}(k,\check \xi)\right)(\xi_1-\xi_2). \label{Demidovich condition 2}
\end{align}
\label{Demidovich Lemma claim}
\end{claim}
\textit{Proof of Claim~\ref{Demidovich Lemma claim}:} To simplify the notations, we denote $a = (f(k,\xi_1) - f(k,\xi_2))$ and $b = \left(\frac{\partial f}{\partial x}(k,\check \xi)\right)(\xi_1-\xi_2)$, we want to show $a^TPa\leq b^TPb$. From \eqref{Demidovich condition 1.5}, we have $a^TPa =a^TPb$. Thus,
\begin{align}
b^TPb-a^TPa &= b^TPb-a^TPa+2(a^TPa-a^TPb)\nn\\
&= (b^TPb-b^TPa) + (a^TPa - a^TPb)\nn\\
&= (a-b)^TP(a-b)  \geq 0
\end{align}
completes the proof of Claim~\ref{Demidovich Lemma claim}. \hfill $\square$

As $J(k,x)\preceq 0$, we see that
\begin{equation*}
  (\xi_1-\xi_2)^T \left(\frac{\partial f}{\partial x}(k,\check \xi)^T P \frac{\partial f}{\partial x}(k,\check \xi)\right)(\xi_1-\xi_2) 
  \leq \rho(\xi_1-\xi_2)^T P(\xi_1-\xi_2)
\end{equation*}
which, with Claim~\ref{Demidovich Lemma claim}, implies that 
\begin{equation}
  \big(f(k,\xi_1) - f(k,\xi_2)\big)^TP  \big(f(k,\xi_1) -f(k,\xi_2)\big)
   \leq \rho(\xi_1-\xi_2)^T P(\xi_1-\xi_2).\label{Demidovich condition 3}
\end{equation}
Define a Lyapunov function candidate $V(k,\xi_1,\xi_2)\coloneqq (\xi_1-\xi_2)^T P(\xi_1-\xi_2)$. Since $P$ is positive definite, $V$ satisfies condition~\eqref{inceremental Lyapunov function 1}. Next, using~\eqref{Demidovich condition 3}, 
\begin{equation*}
  V(k+1,f(k,\xi_1),f(k,\xi_2))-V(k,\xi_1,\xi_2)
  \leq -(1-\rho)V(k,\xi_1,\xi_2),
\end{equation*}
for all $\xi_1,\xi_2\in\R^n$ and $k\in\Z$. This implies that $V$ satisfies condition~\eqref{inceremental Lyapunov function 2}.
Therefore, by virtue of Theorem~\ref{incre Lyapunov characterization}, system~\eqref{system no input} is uniformly globally asymptotically incrementally stable.

It is straightforward to see that the function $V$ defined above also satisfies~\eqref{contraction Lyapunov function 1} and~\eqref{contraction Lyapunov function 2} 
of Theorem~\ref{cont Lyapunov characterization}. Hence, system~\eqref{system no input} is globally contracting.

The additional condition~\eqref{Demidovich bounded condition} together with~\eqref{Demidovich condition 3} implies that system~\eqref{system no input} is uniformly globally (exponentially) convergent by applying \cite[Theorem~1]{PaWo12_TAC}.
$\halmos$

\subsection*{Proof of Theorem~\ref{equivalent contraction vs incremental}}\label{proof of equivalent contraction vs incremental}	

\noindent ``\!\!$\implies$\!\!''\quad
Assume system~\eqref{system no input} is globally uniformly uniformly exponentially incrementally stable. Let $\Big[\begin{smallmatrix}\phi(k;k_0,\xi) \\  \phi_{\delta}(k;k_0,\xi_\delta,\xi)\end{smallmatrix}\Big]$
denote the solution of system~\eqref{original dynamics}--\eqref{displacement dynamics} for the initial condition 
$\Big[\begin{smallmatrix}\xi \\ \xi_\delta\end{smallmatrix}\Big]\in \R^{2n}$, whereby we note that the solutions to the dynamics~\eqref{displacement dynamics} (denoted by $\phi_{\delta}$) depend on the dynamics~\eqref{original dynamics}, so that $\phi_{\delta}$ depends on a third parameter $\xi$ identifying a particular reference trajectory of~\eqref{original dynamics}.
Parametrize the straight line segment connecting $\xi$ and $\xi+\xi_\delta$ at time $k_0$ by a function $\gamma_{k_0}\colon  [0,1]\rightarrow\R^n$  given by
\begin{align}
  \gamma_{k_0} (s)\coloneqq \xi+s\xi_\delta, \label{convex combination}
\end{align}
where $s\in[0,1]$. At time $k\geq k_0$, we denote the parametrized curve initiated from segment~\eqref{convex combination} by
\begin{align*}
  \gamma_k(s)\coloneqq \phi(k;k_0,\xi+s\xi_\delta)=\phi(k;k_0,\gamma_{k_0}(s)).
\end{align*}
We make the following claim:

\begin{claim} The trajectory
  $\Big[\begin{smallmatrix}\phi(k;k_0,\xi) \\  \phi_\delta(k;k_0,\xi_\delta,\xi)\end{smallmatrix}\Big]$ of system~\eqref{original dynamics}--\eqref{displacement dynamics} is identical to $\Big[\begin{smallmatrix}\phi(k;k_0,\xi) \\ \frac{d}{ds}\phi(k;k_0,\gamma_{k_0}(s))\end{smallmatrix}\Big]$; i.e.,~$ \phi_\delta(k;k_0,\xi_\delta,\xi)=\frac{d}{ds}\phi(k;k_0,\gamma_{k_0}(s))$ for all $k,k_0\in \Z$ such that $k\geq k_0$.
\end{claim}
\begin{proof}
  By the chain rule,
  \begin{align*}
    \frac{d}{ds}f(k,\phi(k;k_0,\gamma_{k_0}(s))=\frac{\partial f}{\partial x}(k,x)\frac{d }{ds}\phi(k;k_0,\gamma_{k_0}(s)),
  \end{align*}
  which implies that $\frac{d}{ds}\phi(k;k_0,\gamma_{k_0}(s))$ satisfies the displacement\linebreak[4] dynamics~\eqref{displacement dynamics}. Thus, it is only left to prove that the initial conditions of both trajectories are the same; i.e., $\frac{d}{ds}\phi(k_0;k_0,\gamma_{k_0}(s))=\xi_\delta$. Note that $\frac{d}{ds}\phi(k_0;k_0,\gamma_{k_0}(s))= \frac{d}{ds}\gamma_{k_0} (s) $. Hence, by taking the derivative on both sides of~\eqref{convex combination} with respect to $s$, we obtain the desired result.
\end{proof}

For any sufficiently small $\epsilon>0$, Definition~\ref{exponential incremental stable def} yields $k\geq1$, $\lambda>1$ so that
\begin{equation}
  |\phi(k;k_0,\gamma_{k_0}(s+\epsilon))-\phi(k;k_0,\gamma_{k_0}(s))|
  \leq \kappa |\gamma_{k_0}(s+\epsilon)-\gamma_{k_0}(s)|\lambda^{-(k-k_0)}.\label{exp incre to cont 1}
\end{equation}
Dividing both sides of~\eqref{exp incre to cont 1} by $\epsilon$ then sending $\epsilon \rightarrow 0$, we obtain
\begin{align}
  |\phi_\delta(k;k_0,\xi_\delta,\xi)|=\Big|\frac{d}{ds}\phi(k;k_0,\gamma_{k_0}(s))\Big|\leq \kappa |\xi_\delta| \lambda^{-(k-k_0)}.\label{exp incre to cont 2}
\end{align}
Thus, from the global exponential incremental stability of the\linebreak[4] subsystem~\eqref{original dynamics}, by considering~\eqref{original dynamics}--\eqref{displacement dynamics}, we have established the global exponential stability of the subsystem~\eqref{displacement dynamics}. As this system is linear, its global exponential stability is the same as global exponential incremental stability.

In what follows, we will construct a nonsingular matrix $\Theta(k,x)$ and, hence construct $F(k,x)$ such that~\eqref{contraction condition} is satisfied for all $x \in \mathbb{R}^n$ and $k\in \Z$.

To this end observe that the transfer matrix $\Phi(k,k_{0};\xi)$ of the
linear displacement dynamics~\eqref{displacement dynamics}, which satisfies
$\Phi(k,k_{0};\xi)\xi_{\delta}=\phi(k;k_{0},\xi_{\delta},\xi)$
for all $\xi_{\delta}\in\Rn$ and all $k\geq k_{0}$, satisfies the
exponential bound
\begin{equation}
  \|\Phi(k,k_{0};\xi)\|\leq \kappa\lambda^{-(k-k_{0})}
  \label{eq:1}
\end{equation}
(with $\kappa\geq1$, $\lambda>1$) for $k\geq k_{0}$
\emph{independently} of the initial condition $\xi$ of the
reference trajectory generated by~\eqref{original dynamics}. From here
we may follow, mutatis mutandis, the proof of
\cite[Theorem~23.3]{Rugh95} and define an $n\times n$ matrix
$$
Q(k,\xi)\coloneqq \sum_{j=k}^{\infty}\big(\Phi(j,k;\xi)\big)^{T}\Phi(j,k;\xi)
$$
and noting that, due to~\eqref{eq:1},
$\|Q(k,\xi)\|\leq \frac{\kappa^{2}\lambda^{2}}{\lambda^{2}-1}$
independently of $k$ and $\xi$, so $Q(k,\xi)$ is well
defined. Following the remainder of said proof, one establishes, mutatis mutandis,  that
for all $x \in \mathbb{R}^n$ and $k\in \Z$,
\begin{align}
  &\eta I\preceq Q(k,x) \preceq \rho I,\label{Linear Lyapunov func 1}\\
  \frac{\partial f}{\partial x} (k,x)^T&Q(k+1,x)\frac{\partial f}{\partial x} (k,x)-Q(k,x)\preceq -\nu I\label{Linear Lyapunov func 2}
\end{align}
where $\eta,\rho,$ and $\nu$ are positive constants.

Applying the Cholesky factorization on the uniformly positive definite matrix $Q(k,x)$, there exists a nonsingular matrix $\Theta(k,x)$ such that,  $Q(k,x)=\Theta(k,x)^T\Theta(k,x)$. With this construction of $\Theta(k,x)$, condition~\eqref{bounds on Theta} is automatically satisfied by appealing to~\eqref{Linear Lyapunov func 1}.
Now, multiplying each side of~\eqref{Linear Lyapunov func 2} by $\Theta(k,x)^{-T}$ from the left and $\Theta(k,x)^{-1}$ from the right, we obtain
\begin{multline*}
  \Theta(k,x)^{-T}\frac{\partial f}{\partial x} (k,x)^T\\
   \times \Theta(k+1,x)^T\Theta(k+1,x)\frac{\partial f}{\partial x} (k,x)\Theta(k,x)^{-1}\\
  -\Theta(k,x)^{-T}Q(k,x)\Theta(k,x)^{-1}\\
  \preceq -\nu \Theta(k,x)^{-T}\Theta(k,x)^{-1}
\end{multline*}
which,  with~\eqref{define Jacobian matrix}, implies
\begin{align}
  F(k,x)^TF(k,x) -I\preceq -\nu \Big(\Theta(k,x)\Theta(k,x)^T\Big)^{-1}.\label{construct contraction lyapunov function 0}
\end{align}
Next, since $Q(k,x)\preceq \rho I$ by~\eqref{Linear Lyapunov func 1}, for any $s\in\R^n$ we have
\begin{equation*}
  \big(\Theta(k,x)^{-1}s\big)^T Q(k,x)\big(\Theta(k,x)^{-1}s\big) 
   \leq \rho\big(\Theta(k,x)^{-1}s\big)^T\big(\Theta(k,x)^{-1}s\big).
\end{equation*}
Straightforward manipulations result in
\begin{align*}
  \frac{1}{\rho} s^Ts \leq s^T\big(\Theta(k,x)\Theta(k,x)^T\big)^{-1}s.
\end{align*}
Thus, we arrive at
\begin{align}
  -\nu\big(\Theta(k,x)\Theta(k,x)^T\big)^{-1} \preceq -\frac{\nu}{\rho} I.\label{construct contraction lyapunov function 1}
\end{align}
Combining~\eqref{construct contraction lyapunov function 0} with~\eqref{construct contraction lyapunov function 1}, we see that
\begin{align*}
  F(k,x)^TF(k,x) -I\preceq -\frac{\nu}{\rho}I.
\end{align*}
Therefore, condition~\eqref{contraction condition}
is satisfied for this construction of $F(k,x)$ with $\mu\coloneqq \frac{\nu}{\rho}$, in other words, system~\eqref{system no input} is uniformly globally contracting.

\noindent ``\!\!$\impliedby$\!\!''\quad
Now, assume that~\eqref{system no input} is uniformly globally contracting. Define matrix $Q(k,x)\coloneqq \Theta(k,x)^T\Theta(k,x)$. By the definition of $\Theta(k,x)$, condition~\eqref{bounds on Theta} implies $Q(k,x)$ is uniformly positive definite; i.e., there exists $\eta>0$, $\rho>0$ so that
\begin{align}
  \eta I\preceq Q(k,x)\preceq \rho I.\label{Conveser Linear Lyapunov func 1}
\end{align}
Multiplying each side of \eqref{contraction condition} with $\Theta(k,x)^T$ from the right and $\Theta(k,x)$ from the left, then expanding by using~\eqref{define Jacobian matrix}, we have
\begin{multline*}
  \Theta(k,x)^T \Big(\Theta(k,x)^{-T}\frac{\partial f}{\partial x} (k,x)^T\Theta(k+1,x)^T\\
   \times \Theta(k+1,x)\frac{\partial f}{\partial x} (k,x)\Theta(k,x)^{-1}\Big)\Theta(k,x)\\
  -\Theta(k,x)^T\Theta(k,x) \preceq -\mu\Theta(k,x)^T\Theta(k,x).
\end{multline*}
After straightforward simplifications, we see that
\begin{equation}
  \label{Conveser Linear Lyapunov func 2}
  \frac{\partial f}{\partial x} (k,x)^TQ(k+1,x)\frac{\partial f}{\partial x} (k,x)-Q(k,x)
                                        \preceq -\mu Q(k,x) \preceq -\frac{\eta}{\mu}I
\end{equation}
where the second matrix inequality follows directly from the lower bound of~\eqref{Conveser Linear Lyapunov func 1}. Thus, from~\eqref{Conveser Linear Lyapunov func 1} and~\eqref{Conveser Linear Lyapunov func 2}, applying \cite[Theorem~23.3]{Rugh95}, we conclude that the linear time-varying subsystem~\eqref{displacement dynamics} is uniformly exponentially stable. Furthermore, from the first inequality of~\eqref{Conveser Linear Lyapunov func 2}, we have
\begin{align}
  \frac{\partial f}{\partial x} (k,x)^T&Q(k+1,x)\frac{\partial f}{\partial x} (k,x)\preceq \beta Q(k,x)\label{Conveser Linear Lyapunov func 3}
\end{align}
for some $\beta \in (0,1)$.

Pick any  $\xi_1,\xi_2\in\R^n$. Consider a straight line segment connecting $\xi_1,\xi_2$ and parametrized by a function $\gamma_{k_0}\colon  [0,1]\rightarrow\R^n$, at $k_0$, given by
\begin{align}
  \gamma_{k_0} (s)\coloneqq s\xi_1+(1-s)\xi_2, \label{curve at k=k_0}
\end{align}
where $s\in[0,1]$. The length of this segment at $k=k_0$ is $l_0=|\xi_1-\xi_2|$. Then, for any $k> k_0$,
\begin{align}
  &\gamma_{k} (s)\coloneqq \phi(k;k_0,\gamma_{k_0}(s))\label{curve at k=k}
\end{align}
is a curve connecting $\phi(k;k_0,\xi_1)$ to $\phi(k;k_0,\xi_2)$ parametrized by $s\in[0,1]$.  The length of the curve defined in~\eqref{curve at k=k} at time $k$ is given by
\begin{align*}
  l(k)=\int_{0}^{1}\sqrt{\frac{d}{ds}\gamma_k(s)^TQ(k,\gamma_k(s)) \frac{d}{ds}\gamma_k(s)}ds.
\end{align*} 
Applying the chain rule to $\gamma_{k+1}(s)$ for $s\in[0,1]$ yields
\begin{align*}
  &\frac{d}{ds}\gamma_{k+1}(s)=\frac{d}{ds}\phi(k+1;k_0,\gamma_{k_0}(s))\\
  &=\frac{d}{ds}f(k,\phi(k;k_0,\gamma_{k_0}(s)))\\
  &=\frac{\partial f}{\partial x}(k,x) \frac{d}{ds}\phi(k;k_0,\gamma_{k_0}(s))=\frac{\partial f}{\partial x}(k,x) \frac{d}{ds}\gamma_{k}(s).
\end{align*}
We see that
\begin{multline*}
  \frac{d}{ds}\gamma_{k+1}(s)^TQ(k+1,\gamma_{k+1}(s)) \frac{d}{ds}\gamma_{k+1}(s) \\
  =\Bigg(\frac{\partial f}{\partial x}(k,x) \frac{d}{ds}\gamma_{k}(s)\Bigg)^TQ(k+1,\gamma_{k+1}(s))\\
  \times\Bigg(\frac{\partial f}{\partial x}(k,x) \frac{d}{ds}\gamma_{k}(s)\Bigg)= \frac{d}{ds}\gamma_{k}(s)^T \times\\
  \Bigg[\frac{\partial f}{\partial x}(k,x)  ^TQ(k+1,\gamma_{k+1}(s))\frac{\partial f}{\partial x}(k,x)\Bigg]  \frac{d}{ds}\gamma_{k}(s)\\
  \leq \frac{d}{ds}\gamma_{k}(s)^T \beta Q(k,\gamma_k(s))\frac{d}{ds}\gamma_{k}(s),
\end{multline*}
where the final inequality follows from~\eqref{Conveser Linear Lyapunov func 3}.
Therefore,
\begin{multline*}
  l(k+1)-l(k)\\
  = \int_{0}^{1}\sqrt{\frac{d}{ds}\gamma_{k+1}(s)^TQ(k+1,\gamma_{k+1}(s)) \frac{d}{ds}\gamma_{k+1}(s)}ds\\
  -\int_{0}^{1}\sqrt{\frac{d}{ds}\gamma_{k}(s)^TQ(k,\gamma_k(s)) \frac{d}{ds}\gamma_{k}(s)}ds\\
  \leq(\sqrt{\beta}-1) l(k)
\end{multline*}
which, in turn, implies
$l(k+1)\leq  \frac{l(k)}{\lambda}$,
where $\lambda=\frac{1}{\sqrt{\beta}}>1$ since $\beta \in (0,1)$. Consequently, it is straightforward to see that
$l(k)\leq l_0\lambda^{-(k-k_0)}$
and, hence,
\begin{equation*}
  |\phi(k;k_0,\xi_1)-\phi(k;k_0,\xi_2)|\leq l(k)
  \leq l_0\lambda^{-(k-k_0)}=|\xi_1-\xi_2|\lambda^{-(k-k_0)}
\end{equation*}
for all $\xi_1,\xi_2\in\mathbb{R}^n$ and $k,k_0\in \Z$ such that $k\geq k_0$. We conclude that~\eqref{system no input} is globally exponentially incrementally stable. $\halmos$

\subsection*{Proof of Theorem~\ref{Convergence vs contraction theorem}}\label{proof of Convergence vs contraction theorem}	

The proof of Theorem~\ref{Convergence vs contraction theorem} is a consequence of Theorem~\ref{Theorem Yakubovich lemma application} and Theorem~\ref{equivalent contraction vs incremental}. In particular, by Theorem~\ref{equivalent contraction vs incremental}, we see that a globally contracting system is globally exponentially incrementally stable.  Then, by Theorem~\ref{Theorem Yakubovich lemma application}, we conclude that a globally contracting system whose state  evolves in a compact and positively invariant set, is uniformly globally convergent. $\halmos$

\end{document}